\documentclass[12pt]{amsart}

\pdfoutput=1

\usepackage[english]{babel}

\usepackage{amssymb,latexsym,amsmath,amsthm,graphicx}

\newtheorem{thm}{Theorem}
\newtheorem{lemma}[thm]{Lemma}
\newtheorem{cor}[thm]{Corollary}
\newtheorem{prop}[thm]{Proposition}

\theoremstyle{remark} 
\newtheorem{remark}[]{Remark}

\usepackage{xcolor}

\newcommand{\ii}{{\rm i}}
\newcommand{\ee}{{\rm e}}

\newcommand{\p}{\partial}

\newcommand{\RR}{\mathbb{R}}
\newcommand{\EE}{\mathbb{E}}

\newcommand{\CC}{\mathbb{C}}

\newcommand{\ol}{\overline}
\newcommand{\re}{\text{Re}}
\newcommand{\im}{\text{Im}}
\newcommand{\Ind}{\text{Ind}}

\begin{document}
\author[S-Y. Lee, A. Lerario, E. Lundberg]{Seung-Yeop Lee, Antonio Lerario, and Erik Lundberg}
\title{Remarks on Wilmshurst's theorem}

\begin{abstract}
We demonstrate counterexamples to Wilmshurst's conjecture on the valence of harmonic polynomials in the plane,
and we conjecture a bound that is linear in the analytic degree for each fixed anti-analytic degree.
Then we initiate a discussion of Wilmshurt's theorem in more than two dimensions,
showing that if the zero set of a polynomial harmonic field is bounded then it must have codimension at least two.
Examples are provided to show that this conclusion cannot be improved.

\end{abstract}

\maketitle

\section{Introduction}

Suppose $F: \RR^d \rightarrow \RR^d$ is a vector field with polynomial components, each of degree $n$.
For a generic choice of $F$, intersection theory (Bezout's theorem) implies:
\begin{equation}\label{eq:Bezout}
N_F \leq n^d,
\end{equation}
where $N_F$ is the number of zeros of $F$, points where $F$ vanishes.
This bound is sharp in general, but it is natural to ask:

\noindent {\bf Question:} For interesting special classes of $F$, can (\ref{eq:Bezout}) be improved?

\noindent {\bf Example 1:} 
If $F: \RR^d \rightarrow \RR^d$ preserves orientation, and has no singular zeros, then letting $\deg F$ denote the topological degree, for a generic choice of $F$, 
 $$N_F = \deg F.$$
 For $d=2$ and $F(z)$ an analytic polynomial of $z \in \CC \cong \RR^2$, 
 $$N_F = \deg F = n.$$ 
 This is the fundamental theorem of algebra.
 \begin{remark}
 If $F$ is coercive, its topological degree coincides with the degree of its extension to the one point compactification $S^d=\RR\cup\{\infty\}$. 
 In this case let $y\in \RR^d$ be a regular value for $F$ (Sard's lemma implies the generic $y$ is a regular value); then $\deg F=\sum_{x\in F^{-1}(y)}\textrm{sign}( JF_x)$ (coercivity implies this sum is finite). 
 \end{remark}
 \noindent {\bf Example 2:} Suppose $F(x) = F_+(x) + F_-(x)$ can be decomposed into an orientation preserving part $F_+$ of (topological) degree $n$, and an orientation reversing part $F_-$ of degree $m < n$.
  Can the bound (\ref{eq:Bezout}) be improved?  What if the components of $F$ are assumed to be harmonic polynomials?

While considering special classes of polynomial vector fields, another question is whether the word ``generic'' can be removed
for $F$ within that class.
A. S. Wilmshurst \cite{Wilm94, Wilm98} considered the case of harmonic fields.
Using complex variable notation $z \in \CC$, a harmonic field can be expressed as a sum of analytic and anti-analytic parts:
$$F(z) = p(z) + \ol{q(z)}$$
(this is an instance of a decomposition $F(x) = F_+(x) + F_-(x)$ as described above).
In these terms, Wilmshurst showed:
\begin{thm}[A. S. Wilmshurst, 1994]
If $\deg p = n > m = \deg q$, then $N_F \leq n^2$.
\end{thm}
In other words, for this class of vector fields $F$, Bezout's bound $N_F \leq n^2$ applies generally, not just generically.
This was independently shown in \cite{PSch1996}.

As to the Question of improving (\ref{eq:Bezout}) given additional information, 
Wilmshurst made the tantalizing conjecture that
\begin{equation}\label{eq:Wilmshurst}
N_F \leq 3n-2 + m(m-1).
\end{equation}
This conjecture is stated in \cite[Remark 2]{Wilm98}, and was discussed further in \cite{Sh2002}.
It is also mentioned in the list of open problems in \cite{BL2010}.
For $m=n-1$ the upper bound follows from Wilmshurst's theorem,
and examples were also given in \cite{Wilm98} showing that this bound is sharp (shown independently in \cite{BHS1995}).
For $m=1$, the upper bound was proved by D. Khavinson and G. Swiatek \cite{K-S} (a subsequent extension to rational functions \cite{K-N, K-N2} settled a conjecture in gravitational lensing),
and a proof of the Crofoot-Sarason conjecture given in \cite{G2008} (cf. \cite{BL2004}) established that this bound is sharp.
For $m=n-3$, the conjectured bound is
$$3n-2 + m(m-1) = n^2 - 4n + 10.$$
We provide counterexamples for which
$$N_F > n^2 - 3 n + \mathcal{O}(1).$$
We state the counterexamples in Section \ref{sec:counter} and prove this estimate on the number of zeros.

In Section \ref{sec:gen}, we give an alternative proof of 
Wilmshurst's theorem that relies more heavily on real algebraic geometry
and readily generalizes to harmonic vector fields in higher dimensions but with a weaker conclusion:
the zero set has codimension at least two (for $d=2$ this implies the number of zeros is finite).
As we show by example, this cannot be improved without adding additional assumptions on $F$.
Even though the number of zeros may not be finite, 
the number of connected components is, 
and the bound due to J. Milnor \cite{Milnor} can be applied to estimate the number of connected components (see Section \ref{sec:gen}).

It would be interesting to carry out a random study for $d>2$ similar to what was done in $d=2$ dimensions \cite{LiWei2009}.
For any reasonable distribution of probability on the space of harmonic polynomials, the number of zeros
$N_F$ is finite with probability $1$.  This is because a generic $F$ does satisfy the Bezout bound (see Section \ref{sec:gen}, Proposition \ref{prop:generic}).
Thus, it makes sense to ask what is the expected number of zeros of a random harmonic polynomial field $F$.
To state a concrete problem, take a basis $\{Y_{k,i}(x)\}_{i \in I_k}$ for homogeneous harmonics of degree $k$ 
(which can be expressed explicitly using Gegenbauer polynomials),
where the size of the index set $I_k$ depends on $k$ and $d$.
Let $F$ be a random vector field with $j$th component:
\begin{equation}\label{eq:random}
F_j(x) = \sum_{k=0}^{n}\sum_{i \in I_k} \xi_{k,i,j} \cdot Y_{k,i}(x).
\end{equation}
{\bf Problem:} Calculate or provide large $n$ asymptotics for the expectation $ \EE N_F .$

Note that the inside sum of (\ref{eq:random}) within a single component of $F$ is, 
when restricted to the unit sphere, 
a random spherical harmonic with the same distribution of probability that was studied in
\cite{NazSod2009, LL2012}.

Returning to the deterministic setting, 
the spirit of Wilmshurst's conjecture --- that the maximum number of zeros is linear in $n$ for each fixed $m$ --- still seems plausible.
We do not venture a conjecture as to the exact maximum (and perhaps it is not described by a simple formula),
but we are willing to make the following conjecture.

\noindent {\bf Conjecture:} Let $N_-$ denote the number of orientation reversing zeros of $F(z) = p(z) + \ol{q(z)}$,
$$ N_- \leq m(n-1). $$

In the appendix, we review basic degree theory (the generalized argument principle for harmonic functions)
showing that the Conjecture implies that $N_F \leq 2m(n-1) + n$ when $F$ is free of singular zeros.
Note that this is worse than the Bezout bound for large $m$, but is linear in $n$ for fixed $m$ and reduces to $3n-2$ when $m=1$.

It would also be interesting to further extend the study for $d>2$, and to consider non-harmonic polynomials as well.
We conclude the introduction by giving an illustrative case motivated by \cite{K-S}.
Suppose $F(x) = F_+(x) + A x$, where $F_+$ is orientation preserving and $A$ is an orientation-reversing, non-singular matrix.
If $F_+$ has topological degree $\deg F_+ = n_T$.
Then we can show that
\begin{equation}\label{eq:KS}
N_F \leq n_T^2.
\end{equation}
This follows from one iteration of the fixed-point equation
$$ A^{-1} F_+(x) = x,$$
which implies
$$ A^{-1} F_+(A^{-1} F_+(x) ) = x \implies A^{-1} F_+(A^{-1} F_+(x) ) - x =0. $$
The latter is orientation preserving and has topological degree $n_T^2$.

For $d=2$ and $F$ a complex-analytic function of $z \in \CC$, the Khavinson-Swiatek theorem gives the improved bound 
$$ N_F \leq 3 n - 2.$$

This raises the question whether it is possible to improve (\ref{eq:KS}) to a linear bound when $d>2$ (and perhaps in some cases when $F$ is not harmonic).

\section{Counterexamples to Wilmshurst's conjecture}\label{sec:counter}
 
 \begin{thm}\label{thm:counter}
For $m=n-3$ and $n\geq 4$, there exists polynomials $p(z)$ and $q(z)$ of degrees $n$ and $m$ respectively, such that the number of roots to the equation, $\overline{p(z)}=q(z)$, is at least
\begin{equation}\label{maineq}
n^2-4n+4\,\bigg\lfloor\frac{n-2}{\pi}\arctan \frac{\sqrt{n^2-2n}}{n}\bigg\rfloor+2.
\end{equation}
\end{thm}
\begin{remark}
As $n$ goes to $\infty$, the above expression is $n^2-3n + {\mathcal{O}}(1)$ and, therefore, 
provides infinitely many counterexamples to Wilmshurst's conjecture, with a discrepency that grows linearly in $n$.
\end{remark}
\begin{remark}
We note that \eqref{maineq} is {\em not} the maximal number of roots.  
In fact, the explicit examples used in the proof of the theorem have more roots than stated in \eqref{maineq}.  
However, we leave the theorem with the modest estimate since our main interest is to provide counterexamples, 
and the proof becomes technical with only a slight improvement.  
We also do not know whether our examples attain the maximal number of roots. 
\end{remark}
  
The explicit construction of counterexamples will be done similarly as the construction of the examples for $m=n-1$ by Wilmshurst [Wilm95].

Let us consider the polynomial of degree $n$ given by
\begin{equation}
	f(z)=(z-a)^{n-2} P(z),\quad P(z)=z^2+(n-2)\,a\,z + \frac{(n-2)(n-1)}{2}a^2.
\end{equation}
The quadratic factor is chosen such that $f(z)=z^n + {\mathcal{O}}(z^{n-3})$ as $z\to\infty$.

We define the level lines of ${\rm Im}\,f$ by
\begin{equation}
	\Gamma=\{z\,|\,{\rm Im}f(z)=0 \}.
\end{equation}
At the intersection of $\Gamma$ with the set $\{z\,|\,{\rm Re}\, z^n=0\}$, the following equation is satisfied.
\begin{equation}\label{master}
z^n+\overline{z^n}=f(z) - \overline{f(z)} \quad \Longrightarrow\quad -z^n + f(z) = \overline{z^n}+\overline{f(z)}.    
\end{equation} 
The left hand side is a holomorphic polynomial of degree $n-3$, and the right hand side is a polynomial of degree $n$.  Therefore the intersection points are the roots of the equation $\overline{p(z)}=q(z)$ where $q(z)=-z^n+f(z)$ and $p(z)=z^n+f(z)$.

The following lemma shows properties of $\Gamma$ that are useful in counting the intersection points. 
\begin{lemma}\label{lem:level} The level set $\Gamma$ satisfies the following:
\begin{itemize}
\item $\Gamma$ consists of $n$ smooth (algebraic) curves that starts from $\infty \times \ee^{\ii\pi j /n}$ for some $j\in\{0,1,\cdots,2n-1\}$ and ends at $\infty \times \ee^{\ii\pi k /n}$ for some $k\in\{0,1,\cdots,2n-1\}$ that is not $j$. 
We call each of these a ``line''.
\item For each $n$ there exists an $a$ such that the only intersections of the above curves are at $a$ where $n-2$ lines intersect.  
\end{itemize}
\end{lemma}

\begin{proof} 
The asymptotic behavior stated in the first property follows from the leading order term for $f(z)$.

Intersection of the smooth curves in $\Gamma$ can occur only at a critical point.  The critical points that are not $a$ are given by the roots of the equation
\begin{equation}\label{eq:crit}
\begin{split}0&=f'(z)/(z-a)^{n-3}
\\&=	(n-2)\bigg(z^2+(n-2)\,a\,z + \frac{(n-2)(n-1)}{2}a^2 \bigg)+\big(2z+(n-2)\,a\big)(z-a)
	\\ &=n\, z^2+(n^2-3n) \,a \,z+ \frac{(n-2)n(n-3)}{2}a^2.
	\end{split}
\end{equation}
When ${\rm Im}\,a=0$ (i.e. $a$ is purely real) the above two critical points are not on the real axis and complex conjugate to each other because the discriminant is negative, i.e.
\begin{equation}
	(n^2-3n)^2-2n^2(n-2)(n-3)<0.
\end{equation}
These two critical points, say $\zeta_\pm$ with $\overline\zeta_+=\zeta_-$, are given by 
\begin{equation}
  	\zeta_\pm = \left(-\frac{n-3}{2}\pm \ii\sqrt{\frac{(n-3)(n-1)}{4}}\right)\,a.
  \end{equation}
For purely real $a$, they are either both on $\Gamma$ or none on $\Gamma$ by symmetry.   To find out the correct case, one may calculate the argument of
\begin{equation}
  	f(\zeta_+)=(\zeta_+-a)^{n-2} \big(\zeta_++(n-2)a\big)
\end{equation}  
where the second factor is obtained by simplifying $P(\zeta_+)$ using the equation \eqref{eq:crit} for the critical points.  It is given by
\begin{equation}\begin{split}
\arg f(\zeta_+)&=(n-2) \arctan\left(\frac{\sqrt{(n-3)(n-1)}}{-n+1}\right) +\arctan	\left(\frac{\sqrt{(n-3)(n-1)}}{n-1}\right)
\\&=(n-3) \arctan\left(\frac{\sqrt{(n-3)(n-1)}}{-n+1}\right).\end{split}
\end{equation}
If $f(\zeta_+)$ is purely real then one should have
\begin{equation}\label{eq:tan}\tan \frac{\pi k}{n-3} = \frac{\sqrt{(n-3)(n-1)}}{-n+1}
\end{equation}
for some integer $k$. 

Writing this equation using sine, we have
$$\sin^2 \frac{\pi k}{n-3} = \frac{n-3}{2n-4}.$$
By the double angle identity
\begin{equation}\label{eq:Niven}
  \cos \frac{2 \pi k}{n-3} = 1 - 2\sin^2 \frac{\pi k}{n-3} = \frac{1}{n-2} .
\end{equation}
Let $x = \frac{2 \pi k}{n-3} $,
and suppose (\ref{eq:Niven}) has an integer solution $k$.
Then $x / \pi$ and $\cos x$ are both rational, so Niven's theorem \cite[Corollary 3.12]{Niven} applies
stating that $x = 0, \pm 1/2, \pm 1$.
It is easy to see that none of these possibilities can be realized in (\ref{eq:Niven}).

This shows that, when $a$ is purely real, the only critical points on $\Gamma$ occur at $a$,
but a continuous perturbation of $a$ to the complex plane preserves this property, 
since both the critical points and $\Gamma$ move continuously over the variation of $a$.

\end{proof}

\begin{figure}[ht]
\begin{center}
\includegraphics[width=0.4\textwidth]{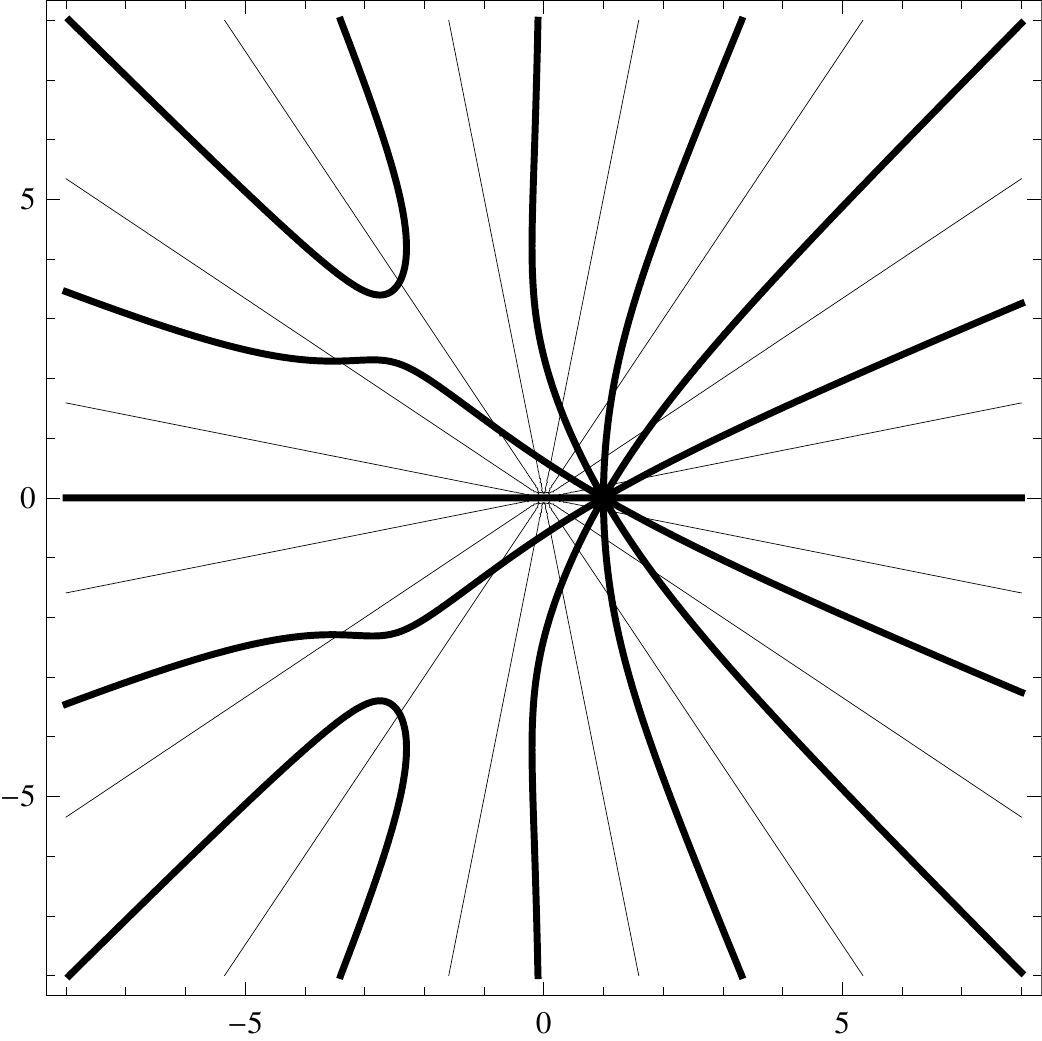}
\includegraphics[width=0.4\textwidth]{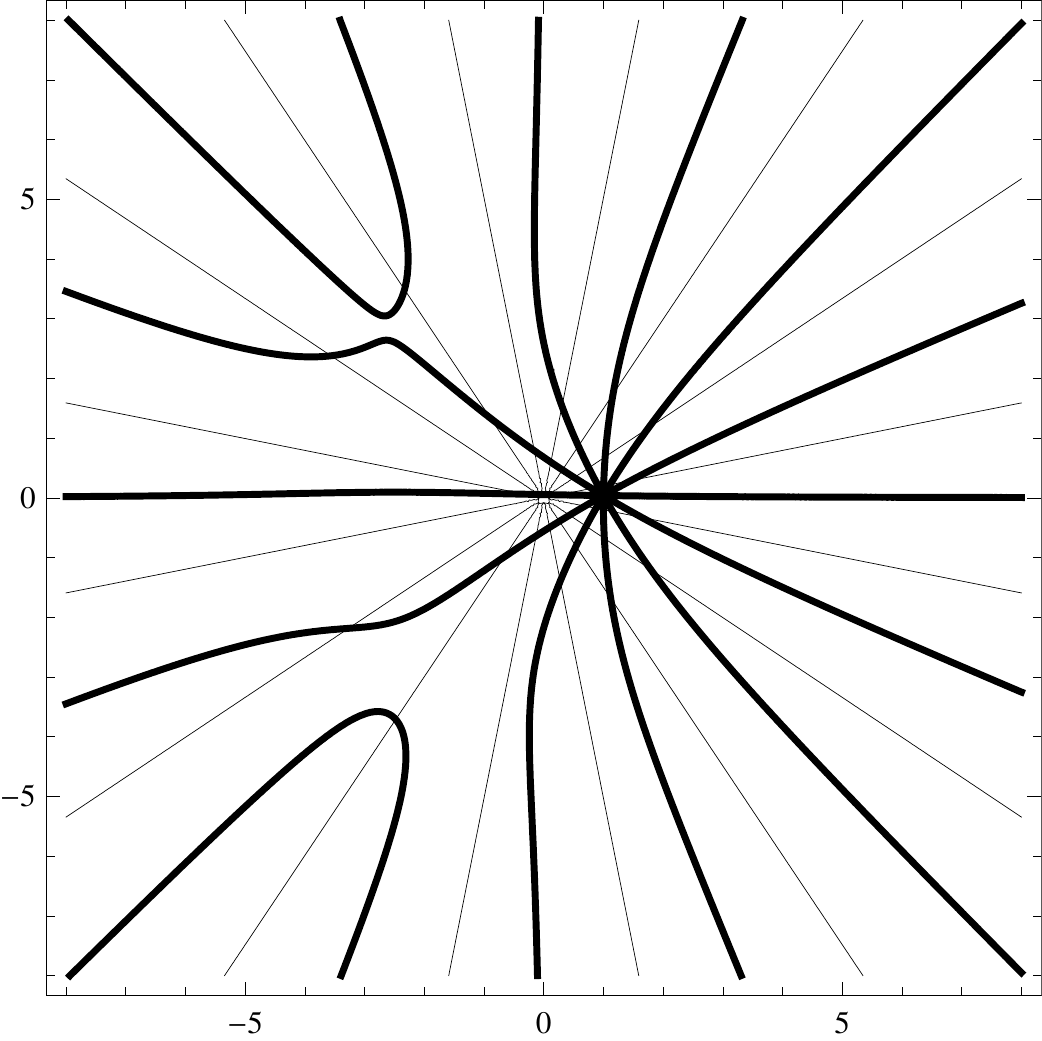}
\caption{\label{fig:n8} $n=8$. The left is for $a=1$; the right is for $a=1+ 0.04\ii$}
\end{center}
\end{figure}
From Lemma \ref{lem:level} one easily observes that, for $a$ purely real,
\begin{itemize}
\item $\Gamma$ has the symmetry with respect to the real axis;
\item the real axis belongs to $\Gamma$;
\item there are exactly two (symmetric) smooth curves that do not go through the point $a$ and they start and end at the consecutive sectors, i.e. $\infty\times\ee^{\ii\pi j/n}$ and $\infty\times\ee^{\ii\pi (j+1)/n}$ for some $j$. 
\end{itemize}
The resulting configuration is in the left picture of Figure \ref{fig:n8}.  
There are two lines that do not pass $a$.  Let us name them each by $\Gamma_+$ and $\Gamma_-$. Each line must contain exactly one root of $f$ (because $\Gamma$ are steepest descent lines of ${\rm Re} \,f$ and ${\rm Re}\,f$ is monotonic over $\Gamma_\pm$), therefore $\Gamma_\pm$ must pass through the two roots of $f$ that are not $a$; let us denote the two roots by
\begin{equation}
	z_\pm=-a \bigg(\frac{n-2}{2} \mp \ii \frac{\sqrt{n^2-2n}}{2}\bigg).
\end{equation}
The following fact shows how these two lines, $\Gamma_\pm$, behave near $\infty$, namely, how many lines pass in between $\Gamma_\pm$.  (For example, in Figure \ref{fig:n8} there are $3$ lines that pass in between $\Gamma_\pm$.) 
\begin{lemma}\label{lem:largest} We say that a smooth curve ``passes between $z_-$ and $z_+$'' if the curve intersects the straight path from $z_-$ to $z_+$ (the straight path does not include the points $z_-$ and $z_+$) exactly an odd
number of times without any tangential intersection.
Let $K$ be the number of smooth curves in $\Gamma$ that pass between $z_+$ and $z_-$. Then $K$ satisfies
\begin{equation}\label{eq:K}
K\geq 2\,\bigg\lfloor\frac{n-2}{\pi}\arctan \frac{\sqrt{n^2-2n}}{n}\bigg\rfloor-1.
\end{equation}
\end{lemma}
\begin{proof}
Let us consider the variation of 
\begin{equation}
	\arg f(z) =	{\rm Im} \Big[(n-2) \log (z-a)\Big] + {\rm Im}\Bigg[\log \bigg(z^2+(n-2)\,a\,z + \frac{(n-2)(n-1)}{2}a^2 \bigg)\Bigg].
\end{equation}
The variation of the second term is zero when measuring along the straight path that connects $z_+$ and $z_-$.  So we have
\begin{equation}
\begin{split}
 \Delta\arg f= 2(n-2)(\pi-\arg(z_--a))&=  2(n-2)\bigg(\pi-\arg \frac{a}{2}\Big(-n +\ii \sqrt{n^2-2n}\Big)\bigg)
  \\&=2(n-2)\arctan \frac{\sqrt{n^2-2n}}{n},
  \end{split}
\end{equation}
where $\Delta\arg f$ is the angular variation of $f$ from $z_+$ to $z_-$.  

Let $z_0$ be the intersection of straight path between $z_-$ and $z_+$ with the real axis.
The number of smooth curves in $\Gamma$ that pass between $z_0$ and $z_+$ is at least 
\begin{equation}
 \bigg\lfloor\frac{\Delta \arg f}{2\pi} \bigg\rfloor-1.
\end{equation}  
The minus 1 above comes because there can be an intersection of the vertical path between $z_0$ and $z_+$ with $\Gamma_+$ at some point that is not $z_+$.  (This in fact occurs at $n=19$.)  

Counting the real axis, the total number of smooth curves in $\Gamma$ that pass between $z_+$ and $z_-$ is at least
\begin{equation}
  	1+2\,\left(\bigg\lfloor\frac{\Delta \arg f}{2\pi} \bigg\rfloor-1\right)=2\,\bigg\lfloor\frac{\Delta \arg f}{2\pi} \bigg\rfloor-1,
\end{equation}  
and this proves our lemma.
\end{proof}

Now let us consider a small continuous perturbation of $a$, see the right picture in Figure \ref{fig:n8}.
For ${\rm Im}\,a$ close to zero, the above $K$ is an odd number (by the symmetry under $x$ axis).   By a small perturbation of $a$, $\Gamma$ no longer contains the origin.  

Finally, we count the intersection points between $\Gamma$ and $\Gamma_0=\{z\,|\,{\rm Re}\,z^n=0\}$.  
Let us call each smooth curve that starts at $a$ and escapes to infinity a ``ray''.   
We count the intersections of each ray with $\Gamma_0$.

Lemma \ref{lem:largest} says that there is a ray that escapes to $\infty\times(-\ee^{\pi\ii j /n})$ for all $j$ in $-n< j\leq n$ except for
\begin{equation}
  j=\frac{K-1}2+1\quad \text{and}\quad   j=\frac{K-1}2+2.
\end{equation}
The ray that starts at $a$ and escapes to $\infty\times(-\ee^{\pi\ii j /n})$ intersects the set $\Gamma_0$ at least
$n-j$ times.  Therefore the total number of intersections is given by the sum of intersections for all rays:
\begin{equation}
  \bigg(n+ 2\sum_{j=1}^{n-1} (n-j)\bigg) - 2\bigg(n-\frac{K-1}{2}-1\bigg)-2\bigg(n-\frac{K-1}{2}-2\bigg)=n^2 - 4n+2K+4.
\end{equation}
Using Lemma \ref{lem:largest} the number of roots to \eqref{master} is given at least by
\begin{equation}
\begin{split}
n^2-4n+2K+4
\geq n^2-4n+4\,\bigg\lfloor\frac{n-2}{\pi}\arctan \frac{\sqrt{n^2-2n}}{n}\bigg\rfloor+2.
\end{split}
\end{equation}
This proves Theorem \ref{thm:counter}.

\section{The higher dimensional case $d>2$}\label{sec:gen}

With $x \in \RR^d$, let 
$$F(x_1,x_2,..,x_d) = \langle h_1(x), h_2(x),..,h_d(x) \rangle$$ be a vector field in $\RR^d$ 
having harmonic polynomial components $h_k$ so that the largest degree appearing is $n$.
Let $F = F_n + F_L$ be a decomposition of $F$ into a vector field $F_n$ containing the leading homogeneous terms and $F_L$ containing the lower order terms.
Using this set up, Wilmshurst's theorem can be stated in a form that is free of complex variable notation.

\begin{thm}\label{thm:W}
Let $d=2$. For a harmonic polynomial vector field $F(x)=F_n(x) + F_L(x)$, with $F_n$ the vector field of leading degree terms, if $F_n$ does not vanish on the unit circle $S^{1}$ 
then $N_F$ is finite and $N_F \leq n^2$.
\end{thm}
This immediately implies the theorem as stated in the introduction,
since the leading part of the harmonic field 
$$F(x,y) = p(z) + \ol{q(z)} = \left( \re \{ p(x+iy) + q(x+iy) \}, \im \{ p(x+iy) - q(x+iy) \} \right),$$
is, for some nonzero constant $c_n$,
$$F_n(x,y) = c_n \left( \re \{ (x+iy)^n \}, \im \{ (x+iy)^n \} \right),$$
and this does not vanish on the unit circle $S^1$.

This naturally leads to the question of whether Theorem \ref{thm:counter} is true in $d>2$ dimensions.
We give a generalization showing that the zero set of $F$ has codimension at least two (which implies $N_F$ is finite when $d=2$).
Examples with $d>2$ and codimension exactly two are described below.

First we recall a definition of \emph{dimension} for an algebraic subset of $\mathbb{R}^d$ 
(in fact the same definition applies to the wider class of semialgebraic subsets). 
The definition we give is taken from \cite{BCR} and is as close as possible to the reader's intuition.

Recall that given a compact semialgebraic subset $S\subset \mathbb{R}^d$ there exists a semialgebraic triangulation $\phi:|K|\to S$, 
where $|K|$ is the total space of a simplicial complex $K\subset \mathbb{R}^n$ and $\phi$ is a continuous semialgebraic map 
(see Theorem 9.2.1 from \cite{BCR})

We define the dimension of $S$ to be the dimension of $K$: it is the maximum over the dimension of the simplices in $K.$

In fact it is even possible to stratify $S=\coprod_{\alpha=1}^s S_\alpha$ in such a way that each stratum $S_\alpha$ is a smooth manifold and the dimension of $S$ equals $\max\{\dim (S_\alpha)\}$, 
where $\dim S_\alpha$ is the dimension as a smooth manifold (still referring to \cite{BCR}, 
Proposition 9.1.8 gives the stratification and Theorem 9.1. ensures the definitions of dimension agree).

\begin{thm}\label{thm:gen}
Suppose a real algebraic set $X \subset \RR^d$ is defined by harmonic polynomials.
If $X$ is bounded then the codimension of $X$ in $\RR^d$ is at least two.
\end{thm}

\begin{cor}[Generalization of Wilmshurst's theorem]
\label{cor:gen}
For a harmonic polynomial vector field $F(x)=F_n(x) + F_L(x)$ as above, if $F_n$ does not vanish on the unit sphere $S^{d-1}$ 
then the zero set $\{x \in \RR^d: F(x) = 0 \}$ is of codimension at least two.
\end{cor}

For $d=2$ this reduces to Theorem \ref{thm:W}.

\begin{proof}[Proof of Corollary \ref{cor:gen}]
The assumption on $F_n$ implies that $\{x \in \RR^d : F(x) =0 \}$ is bounded so that Theorem \ref{thm:gen} applies.
Indeed, for all $x \in \RR^d$ with $|x|$ large enough,
$$|F(x)| \geq |F_n(x)| - |F_L(x)| \geq |x|^n \min_{\theta \in S^{d-1}}|F_n(\theta)| - |x|^{n-1}\max_{\theta \in S^{d-1}}|F_L(\theta)| > 0.$$
\end{proof}

\begin{proof}[Proof of Theorem \ref{thm:gen}]
Suppose $X=\{F=0\}$ has a component of codimension one, i.e. $\dim(X)=d-1$. We will prove that this implies:
\begin{equation}\label{c}\mathbb{R}^d\backslash X \quad \textrm{has at least one bounded component $A,$}\end{equation} and this will give an absurdity. In fact, assuming (\ref{c}), pick a component $h=h_i$ of $F=(h_1, \ldots, h_d)$ that doesn't vanish identically on $A$. By possibly replacing $h$ with $-h$, we can assume that $h\geq 0$ on $\textrm{Clos}(A)$ and since $\textrm{Clos}(A)$ is compact ($A$ is bounded), then $h|_{\textrm{Clos}(A)}$ has a maximum $M=h(x_0)>0$, $x_0\in A$. The Maximum Principle for harmonic functions implies that $h$ is constant on $A$, against the assumption that it vanishes on a point $x\in X$ that is also a limit point of $A.$

In order to prove that $\dim (X)=d-1$ implies (\ref{c}) we proceed as follows.
First we consider the \emph{one point compactification} $S^d=\mathbb{R}^d\cup \{\infty\}$ of $\mathbb{R}^d$. Since $X$ is closed and bounded, then the number of connected components of $S^d\backslash X$ is the same of $\mathbb{R}^d\backslash X$; moreover the existence of at least two components of $S^d\backslash X$ implies one of them does not contain $\{\infty\}$ and this component will also be a \emph{bounded} component of $\mathbb{R}^d\backslash X.$ 

To compute the number of connected components of $S^d\backslash X$, recall that this number equals $b_0(S^d\backslash X)=\textrm{rk} H_0(S^d\backslash X; \mathbb{Z}_2)$ (the zero-th homology group with coefficients in the field $\mathbb{Z}_2$; the choice of $\mathbb{Z}_2$ coefficients does not change the rank of $H_0$ and will be useful for the study of the homology of the \emph{real} algebraic set $X$). Moreover, since $X$ is compact, Alexander duality (see \cite{Hatcher}, Theorem 3.44) implies:
$$\tilde{H}_0(S^d\backslash X; \mathbb{Z}_2)\simeq H^{d-1}(X;\mathbb{Z}_2)\simeq H_{d-1}(X;\mathbb{Z}_2)$$
where the first group is the \emph{reduced} zero-th homology group, whose rank equals $b_0(S^d\backslash X)-1$, and the last isomorphism comes from the fact that  homology and cohomology with coefficients in a field are isomorphic.

Let us now triangulate $X$ as for the definition of its dimension (see above). Proposition 11.1.1 from \cite{BCR} says that the sum of the simplices of dimension $d-1$ with coefficients in $\mathbb{Z}_2$ is a cycle in $X$; this cycle cannot be a boundary, since there are no $d$-dimensional simplices. Hence $H_{d-1}(X)\neq 0$
and the conclusion follows from $$b_0(S^d\backslash X)=\textrm{rk}\tilde{H}_0(S^d\backslash X;\mathbb{Z}_2)+1=\textrm{rk}H_{d-1}(X;\mathbb{Z}_2)+1\geq 2.$$
\end{proof}

In this setting, given that ``codimension two'' is equivalent to ``finite'' when $d=2$
yet becomes a weaker conclusion in higher dimensions, it is natural to ask whether
it is possible to show when $d>2$ the stronger conclusion that $N_F$ is finite in Corollary \ref{cor:gen}.
The next example shows that ``codimension two'' cannot be improved.
For simplicity, we describe the example for $d=3$
but a similar construction works in higher dimensions to exhibit a harmonic field satisfying the hypothesis of Corollary \ref{cor:gen}
and having zero set of codimension exactly two.

\smallskip

\noindent {\bf Example:}
For simplicity of notation we consider a vector field in $\RR^3$.
Take
$$F(x,y,z) = \langle u(x,y,z), v(x,y,z), w(x,y,z) \rangle,$$
with the following harmonic polynomials as components
$$u(x,y,z) = xy(6z^2-x^2-y^2),$$
$$v(x,y,z) = (x^2-y^2)(6z^2-x^2-y^2), $$
$$w(x,y,z) = 35(z-1)^4-30(z-1)^2(x^2+y^2+(z-1)^2) + 3(x^2+y^2+(z-1)^2)^2 .$$

The first two are spherical harmonics of fourth degree, and the latter is a spherical harmonic of the same degree, shifted along the $z$-axis.
The leading part $F_4$ of the field consists of $\langle u(x,y,z),v(x,y,z),w_4(x,y,z) \rangle$,
where 
$$w_4(x,y,z) = 35z^4-30z^2(x^2+y^2+z^2) + 3(x^2+y^2+z^2)^2,$$
is the homogeneous fourth-degree part of $w$.

\medskip

\noindent {\bf Claim 1:} The hypothesis of Corollary \ref{cor:gen} is satisfied.

Note that $u$, $v$, and $w_4$ are three elements from the standard basis for spherical harmonics of degree four.
Consider the nodal lines of $u$, $v$, and $w_4$ on the sphere (which are well-studied). 
We will see that there are no points in $S^2$ common to all three.
We first notice that $u$ and $v$ have no meridional lines in common; 
$xy=0$ only intersects $x^2-y^2=0$ at the North and South poles,
where $w_4$ is non-vanishing.
It remains to check that the horizontal nodal lines, which are identical for $u$ and $v$, do not intersect those of $w_4$.
With respect to the angle $\theta$ from the North pole, $w_4$ restricted to the sphere is a constant times $P_4(\cos(\theta))$, where $P_4$ is a Legendre polynomial, 
whereas the zeros with respect to $\theta$ of $u$ and $v$ are given by those of $P_{4,2}(cos(\theta))$, with $P_{4,2}$ the associated Legendre polynomial.
The zeros of these polynomials are well known and none are in common.

\medskip

\noindent {\bf Claim 2:} The zero set of $F$ has codimension exactly two.

The zero set of $u$ consists of a cone $C$ with vertex at the origin along with a pair of orthogonal planes containing the $z$-axis.
The zero set of $v$ has the same description except the pair of orthogonal planes is different (but the cone $C$ is the same).
The zero set of $w$ is a set of two cones $C_1$, $C_2$ each with vertex at the point $(0,0,1)$.
The slopes of $C$, $C_1$, and $C_2$ are all distinct.
This is a consequence of the fact used 
above that $P_4$ and $P_{4,2}$ don't share any zeros.
Indeed, the zeros of $P_4$ determine the slopes of the cones in the zero set of $w$ while the zeros of $P_{4,2}$ determine the slopes for $u$ and $v$.
It follows that each cone of $\{w=0\}$ intersects the cone $C = \{u=v=0\}$ in two circles,
as is the case for any intersection of two cones with common axis but different vertices and different slopes.

\begin{figure}[h]
    \begin{center}
    \includegraphics[scale=.25]{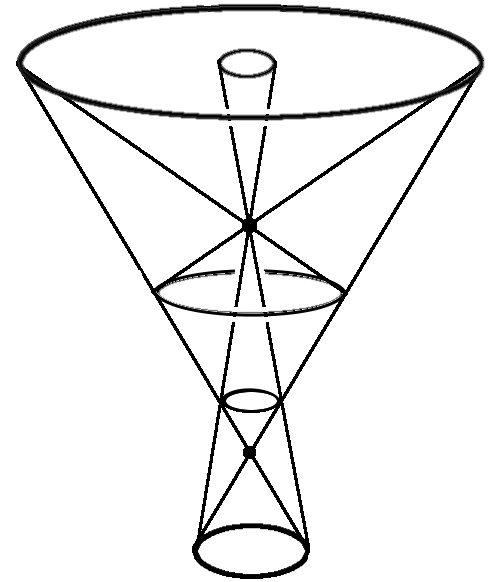}
    \end{center}
    \caption{The zero set of $F$ includes four circles where $C$ intersects the union of $C_1$ and $C_2$.}
    \label{fig:cones}
\end{figure}

\bigskip

This example shows that Bezout's theorem cannot be applied when $d>2$.
However, it is still possible to give an upper bound on the number $b_0(F)$ of \emph{connected components} of $F$.
The bound due to J. Milnor \cite{Milnor} for the total Betti number of an intersection
in particular gives an upper bound on the number of connected components (zeroth Betti number).
Applying this to our situation gives
$$b_0(F) \leq n(2n-1)^{d-1}.$$

We note that in the generic case it is possible to apply the Bezout bound.

\begin{prop}\label{prop:generic}
For a generic harmonic polynomial vector field $F(x)$, 
the number of zeros is finite and bounded by
the product of the degrees of its component functions.
\end{prop}

\begin{proof}
Let $F=(h_1, \ldots, h_d)$ with each $h_i$ of degree $n_i$; if we denote by $H_{l, k}$ the space of real harmonic polynomials of degree $l$ in $k$ variables, in such a way that $H_{l,k}\subset \mathbb{R}[x_1, \ldots,x_k]_l=P_{l,k}$, then we have:
$$F\in H_{n_1, d}\oplus\cdots\oplus H_{n_d, d}=V.$$
In this way $V$ is a finite dimensional subspace of $W=P_{n_1, d}\oplus\cdots P_{n_d, d}.$ An element of $W$ can be considered as a ``system of polynomial equations'', in the sense that given $P=(P_1, \ldots, P_d)\in W$ we can write the system $\{P_1=\cdots=P_d=0\}$. It is well known that for the generic $P\in W$ the above system has a finite number of solutions (being the number of variables equal to the number of equations), and such number is bounded by (this is simply Bezout's theorem):
$$N_P\leq n_1\cdots n_d=N_W$$
Thus, given $F\in V\subset W$, we can perturb it to a $\tilde{F}$ with finitely many zeros, but we do not know (yet) that we can do this perturbation without leaving $V.$ To show that this is indeed possible, consider the ``discriminant'' $\Delta\subset W$:
$$\Delta=\{P\in W\,|\, \textrm{ the system of equations associated to $P$ is degenerate}\}.$$
We claim that $\Delta$ is a semialgebraic set of dimension: $$\dim (\Delta)\leq \dim (W)-1.$$
Semialgebraicity is clear: it is the projection on the third factor of the semialgebraic set $S=\{(x, \eta, P)\in \mathbb{R}^d\times (\mathbb{R}^{d})^*\backslash \{0\} \times W\, |\, P(x)=0, \, d_xP=0\}$ (projections of semialgebraic sets are semialgebraic, see \cite{BCR}). 
The fact that $\Delta$ has codimension one (according to the above definition) is equivalent to the the fact that it doesn't contain any open subsets, which is clear from the genericity of regular systems. 
Consider now the \emph{Zariski} closure $\overline{\Delta}$ of $\Delta$ in $W$: Proposition 2.8.2 of \cite{BCR} implies $\overline{\Delta}$ has the same dimension as $\Delta$, hence $\overline{\Delta}$ is a \emph{proper} algebraic set. 
Notice that if $P\notin \overline{\Delta}$, then $P$ is regular.

Finally consider the intersection $V\cap \overline{\Delta}$: it is an algebraic set and, if proper, its complement will be an \emph{open dense} set of regular elements of $V$. To show that $V\cap \overline{\Delta}$ is proper, it is enough to exhibit \emph{one} regular $F\in V$: then for every $P$ in a neighborhood $U$ of $F$ in $W$ such $P$ will be regular; hence $U\cap V$ will be a nonempty open set of regular elements. 
The existence of a regular $F$ is left as an exercise.
\end{proof}

\section{Appendix}\label{sec:appendix}

Suppose $F(z) = p(z) + \ol{q(z)}$ 
is free of singular zeros 
(zeros where the Jacobian of $F$, $|p'(z)|^2 - |q'(z)|^2$ vanish).
Let $N_+$ count the orientation-preserving zeros and $N_-$ the orientation-reversing zeros of $F$.
Suppose $\Omega$ is a domain with smooth boundary $\p \Omega$ without zeros on $\p \Omega$.
The argument principle for harmonic functions \cite{DHL} states that the winding number $\Ind_{\partial \Omega} F(z)$ around the boundary of a domain $\Omega$
counts the number of orientation preserving zeros inside $\Omega$ minus the number of orientation reversing zeros.
Thus, for a large enough circle $C$:
$$ \Ind_{C} F(z) = N_+ - N_-.$$
If $\deg p = n > m = \deg q$ then $ \Ind_{C} F(z) = n$ for $C$ large enough, since the $z^n$ term dominates.
Thus, $N_+ = N_- + n$.

In the introduction, we conjectured the bound $N_- \leq m(n-1)$ which is linear in $n$.
(The number $m(n-1)$ is the degree of $q$ times the maximum possible number of components of the set where $F$ reverses orientation.)
According to the above, the conjecture implies that $N_+ \leq m(n-1) + n$, and thus $ N_F = N_+ + N_- \leq  2m(n-1) + n.$

\end{document}